\documentclass[11pt]{article}
\usepackage[dvips]{epsfig}
\usepackage{amsmath,amssymb,euscript,graphicx,amsfonts,verbatim}
\usepackage{enumerate,color}
\usepackage{graphicx}
\usepackage{enumitem}
\usepackage{enumerate,color}
\usepackage{graphicx}

\usepackage[colorlinks=true,linkcolor=blue,urlcolor=blue,citecolor=blue]{hyperref}

\newtheorem{theorem}{Theorem}[section]
\newtheorem{proposition}[theorem]{Proposition}
\newtheorem{lemma}[theorem]{Lemma}
\newtheorem{corollary}[theorem]{Corollary}

\newtheorem{conjecture}[theorem]{Conjecture}
\newtheorem{definition}[theorem]{Definition}

\newtheorem{remark}[theorem]{Remark}
\newtheorem{notation}[theorem]{Notation}

\newenvironment{proof}{{\noindent \sc Proof.}}{\hfill $\Qed$\\}

\newcommand{\tA}{\breve{A}}

\newcommand{\Qed}{\rule{2.5mm}{3mm}}

\newcommand{\Ga}{\Gamma}

\newcommand{\CC}{\mathbb{C}}

\DeclareMathOperator{\mat}{Mat}

\newcommand{\MX}{\mat_X(\CC)}
\newcommand{\e}{E^{*}}
\newcommand{\R}{{\cal R}}

\newcounter{case}

\renewcommand{\thecase}{\arabic{case}}

\newcounter{subcase}

\numberwithin{subcase}{case}

\begin{document}


\begin{center}
{\bf\Large Distance-regular graphs with classical parameters that support a uniform structure: case $q \ge 2$} \\ [+4ex]
Blas Fernández{\small$^{a,b}$},  
Roghayeh Maleki{\small$^{a,b,}$\footnote{Corresponding author: roghayeh.maleki@famnit.upr.si}},  
\v Stefko Miklavi\v c{\small$^{a, b,c}$},
\\
Giusy Monzillo{\small$^{a, b}$} 
\\ [+2ex]
{\it \small 
$^a$University of Primorska, UP IAM, Muzejski trg 2, 6000 Koper, Slovenia\\
$^b$University of Primorska, UP FAMNIT, Glagolja\v ska 8, 6000 Koper, Slovenia\\
$^c$IMFM, Jadranska 19, 1000 Ljubljana, Slovenia}
\end{center}


\begin{abstract}
Let $\Ga=(X,\mathcal{R})$ denote a finite, simple, connected, and undirected non-bipartite graph with vertex set $X$ and edge set $\mathcal{R}$. Fix a vertex $x \in X$,  and define $\mathcal{R}_f = \mathcal{R} \setminus \{yz \mid \partial(x,y) = \partial(x,z)\}$, where $\partial$ denotes the path-length distance in $\Ga$. Observe that the graph $\Ga_f=(X,\mathcal{R}_f)$ is bipartite. We say that $\Ga$ supports a uniform structure with respect to $x$  whenever $\Ga_f$ has a uniform structure with respect to $x$ in the sense of Miklavi\v{c} and Terwilliger \cite{MikTer}.

Assume  that  $\Ga$ is a distance-regular graph with classical parameters $(D,q,\alpha,\beta)$ and diameter $D\geq 4$. Recall  that $q$ is an integer such that $q \not \in \{-1,0\}$. The purpose of this paper is to study when $\Ga$ supports a uniform structure with respect to $x$. We studied the case $q \le 1$ in \cite{FMMM}, and so in this paper we assume $q \geq 2$. Let $T=T(x)$ denote the Terwilliger algebra of $\Ga$ with respect to $x$. Under an additional assumption that every irreducible $T$-module with endpoint $1$ is thin, we show that if $\Ga$ supports a uniform structure with respect to $x$, then  either $\alpha = 0$ or $\alpha=q$, $\beta=q^2(q^D-1)/(q-1)$, and $D \equiv 0 \pmod{6}$.

\end{abstract}
\begin{quotation}

\end{quotation}

\begin{quotation}
\noindent 
{\em Math. Subj. Class.:}  05E99, 05C50.
\end{quotation}


\section{Introduction}
\label{sec:intro}

Throughout this paper, all graphs will be finite, simple, connected, and undirected. The definition of a uniform structure for a bipartite distance-regular graph was introduced in \cite{MikTer}, and then generalized to the case of an arbitrary connected bipartite graph in \cite{FMMM}.  Let $\Ga$ denote a connected bipartite graph, fix a vertex $x$ of $\Ga$ and let $T=T(x)$ denote the corresponding Terwilliger algebra with respect to $x$, and let $L=L(x) \in T$ (resp. $R=R(x) \in T$) denote the lowering matrix (resp. raising matrix) of $\Ga$ with  respect to $x$. We say that $\Ga$ has a {\em uniform structure with respect to} $x$ if certain linear dependencies among $RL^2$, $LRL$, $L^2R$, and $L$ are satisfied (see Section \ref{sec:uniform} for more details). Uniform structures of $Q$-polynomial bipartite distance-regular graphs were studied in detail in \cite{MikTer}, where it was shown that except for one special case, a uniform structure is always attained.

Assume now that $\Ga=(X,\mathcal{R})$ is a non-bipartite graph with vertex set $X$ and edge set $\mathcal{R}$. Fix a vertex $x \in X$, and define $\mathcal{R}_f = \mathcal{R} \setminus \{yz \mid \partial(x,y) = \partial(x,z)\}$. Observe that the graph $\Ga_f=(X,\mathcal{R}_f)$ is bipartite. We say that $\Ga$ {\em supports a uniform structure with respect to} $x$ if $\Ga_f$ admits a uniform structure with respect to $x$. The study of non-bipartite graphs that support a uniform structure was initiated by Worawannotai \cite{w:dual}, who essentially showed that (non-bipartite) dual polar graphs support a uniform structure. Later, Hou et al. \cite{hou2018folded} showed that  a folded $(2D+1)$-hypercube supports a uniform structure. 

The main purpose of this paper is to continue the investigation of which non-bipartite $Q$-polynomial distance-regular graphs support a uniform structure. As the complete answer to this problem seems to be currently beyond our reach, we will concentrate on distance-regular graphs with classical parameters (which are all also $Q$-polynomial). Let $\Ga$ denote a distance-regular graph with classical parameters $(D,q,\alpha,\beta)$ with $D \ge 3$. Recall that $q$ is an integer, different from $0$ and $-1$. In our previous paper \cite{FMMM} we assumed that $q \le 1$. We first briefly describe the results obtained in \cite{FMMM}. 

The graph $\Ga$ is said to be of {\em negative type} whenever $q \le -2$. The first main result of \cite{FMMM} concerns graphs of negative type. We showed that if $\Ga$ is of negative type with $D \ge 4$, then $\Ga$ supports a uniform structure with respect to $x$ if and only if $\Ga$ is the dual polar graph $^2 A_{2D-1}(-q)$. The second main result of \cite{FMMM} concerns graphs with classical parameters with $q=1$. We gave a complete classification of these graphs that support a uniform structure with respect to $x$. 

In this paper, we study graphs with classical parameters with $q \ge 2$ and $D \ge 4$. Assume in  addition that every irreducible $T$-module with endpoint $1$ is thin. The main result of this paper is that if $\Ga$ supports a uniform structure with respect to $x$, then  either $\alpha = 0$, or $\alpha=q$, $\beta=q^2(q^D-1)/(q-1)$, and $D \equiv 0 \pmod{6}$. We will study the case  $\alpha=0$ in detail in our future paper.

This paper is organized as follows. In Section \ref{sec:ter}, we recall the definition and some basic properties of the Terwilliger algebra $T$ of a graph $\Ga$ (with respect to a fixed vertex). In Section \ref{sec:uniform}, we give a formal definition of a uniform structure for bipartite graphs, and we prove some supplementary results. The notion of  distance-regular graphs with classical parameters and that of strongly regular graphs are briefly reviewed in Section \ref{sec:drg}. Next, in Sections \ref{sec:locI}, \ref{sec:locII}, and \ref{sec:localgraph}, we review some results on thin irreducible $T$-modules with endpoint $1$, describe their local eigenvalues, and provide more details and results on the local graph of $\Ga$. Finally, in Section \ref{sec:main}, we consider  distance-regular graphs with classical parameters $(D,q,\alpha,\beta)$ with $q \geq 2$ and $D \ge 4$, and prove our main results.


\section{Terwilliger algebra of a graph}
\label{sec:ter}

Let $\Ga=(X,\mathcal{R})$ denote a graph with vertex set $X$ and edge set $\mathcal{R}$. In this section, we recall the definition of a Terwilliger algebra of $\Ga$. Let $\partial$ denote the path-length distance function of $\Ga$. The diameter $D$ of $\Ga$ is defined as $\max\{\partial(x,y) \mid x,y \in X \}$. Let $\MX$ denote the matrix algebra over the field of complex numbers $\CC$, consisting of all matrices whose rows and columns are indexed by $X$. Let $V$ denote the vector space over $\CC$ consisting of column vectors whose coordinates are indexed by $X$. We observe that $\MX$ acts on $V$ by left multiplication. We call $V$ the \emph{standard module}. We endow $V$ with the Hermitian inner 
product $\langle \cdot, \cdot \rangle$ that satisfies $\langle u,v \rangle = u^{\top}\overline{v}$ for 
$u,v \in V$, where $\top$ denotes the transposition and $\overline{\phantom{v}}$ denotes the complex conjugation. For any $y \in X$, let $\widehat{y}$ denote the element of $V$ with $1$ in the ${y}$-coordinate and $0$ in all other coordinates. We observe that $\{\widehat{y}\;|\;y \in X\}$ is an orthonormal basis for $V$. 

Let $A \in \MX$ denote the \emph{adjacency matrix} of $\Gamma$:
$$\left( A\right) _{yz}=
\begin{cases}
	\hspace{0.2cm} 1 \hspace{0.5cm} \text{if} & \partial(y,z)=1,   \\
	\hspace{0.2cm} 0 \hspace{0.5cm} \text{if} &  \partial(y,z) \neq 1
\end{cases} \qquad (y,z \in X).
$$
Let $M$ denote the subalgebra of $\MX$ generated by $A$. The algebra $M$ is called the \emph{adjacency algebra} of $\Gamma$. Observe that $M$ is commutative.
In order to recall the dual idempotents of $\Gamma$, we fix a vertex $x \in X$ for the rest of this section. Let $\varepsilon=\varepsilon(x)$ denote the eccentricity of $x$, that is, $\varepsilon = \max \{\partial(x,y) \mid y \in X\}$. For $ 0 \le i \le \varepsilon$, let $E_i^*=E_i^*(x)$ denote the diagonal matrix in $\MX$ with $(y,y)$-entry given by
\begin{eqnarray*}
	\label{den0}
	(\e_i)_{y y} = \left\{ \begin{array}{lll}
		1 & \hbox{if } \; \partial(x,y)=i, \\
		0 & \hbox{if } \; \partial(x,y) \ne i \end{array} \right. 
	\qquad (y \in X).
\end{eqnarray*}
We call $\e_i$ the \emph{$i$-th dual idempotent} of $\Gamma$ with respect to $x$ \cite[p.~378]{Terpart1}. We observe 
(ei)  $\sum_{i=0}^\varepsilon E_i^*=I$; 
(eii) $\overline{E_i^*} = E_i^*$ $(0 \le i \le \varepsilon)$; 
(eiii) $E_i^{*\top} = E_i^*$ $(0 \le i \le \varepsilon)$; 
(eiv) $E_i^*E_j^* = \delta_{ij}E_i^* $ $(0 \le i,j \le \varepsilon)$,
where $I$ denotes the identity matrix of $\MX$. It follows that $\{E_i^*\}_{i=0}^\varepsilon$ is a basis for a commutative subalgebra $M^*=M^*(x)$ of $\MX$. The algebra $M^*$ is called the \emph{dual adjacency algebra} of $\Gamma$ with respect to $x$ \cite[p.~378]{Terpart1}. Note that for $0 \le i \le \varepsilon$ we have
$\e_i V = {\rm span} \{ \widehat{y} \mid y \in X, \partial(x,y)=i\}$, 
and  
\begin{equation}
	\label{vsub}
	V = E_0^*V + E_1^*V + \cdots + E_\varepsilon^*V \qquad \qquad {\rm (orthogonal\ direct\ sum}). \nonumber 
\end{equation}
The subspace $\e_i V$ is known as the \emph{$i$-th subconstituent of $\Gamma$ with respect to $x$}. 
For convenience, $\e_{i}$ is assumed to be the zero matrix of $\MX$ for $i<0$ and $i>\varepsilon$.

The \emph{Terwilliger algebra of $\Gamma$ with respect to $x$}, denoted by $T=T(x)$,  is the  subalgebra of $\MX$ generated by $M$ and $M^*$ \cite{Terpart1}. Observe that $T$ is generated by the adjacency matrix $A$ and the dual idempotents $E^*_i \, (0\leq i\leq \varepsilon)$, and so it is finite-dimensional. In addition, $T$ is closed under the conjugate-transpose map by construction, and is hence semi-simple. For a vector subspace $W \subseteq V$, we denote by $TW$ the subspace $\{B w \mid B \in T, w \in W\}$. Let us now recall the lowering, the flat, and the raising matrices of $\Ga$. 

\begin{definition} \label{def2} 
	Let $\Gamma=(X,\R)$ denote a graph with diameter $D$. Pick $x \in X$, let $\varepsilon=\varepsilon(x)$ denote the eccentricity of $x$ and  $T=T(x)$ denote the Terwilliger algebra of $\Gamma$ with respect to $x$. Define $L=L(x)$, $F=F(x)$, and $R=R(x)$ in $\MX$ by
	\begin{eqnarray}\label{defLR}
		L=\sum_{i=1}^{\varepsilon}E^*_{i-1}AE^*_i, \hspace{1cm}
		F=\sum_{i=0}^{\varepsilon}E^*_{i}AE^*_i, \hspace{1cm}
		R=\sum_{i=0}^{\varepsilon-1}E^*_{i+1}AE^*_i. \nonumber 
	\end{eqnarray}
	We refer to $L$, $F$, and $R$ as the \emph{lowering}, the \emph{flat}, and the \emph{raising matrix of} $\Ga$ \emph{with respect to $x$}, respectively.
\end{definition}
Note that, by definition, $L, F, R \in T$, $F=F^{\top}$, $R=L^{\top}$, and $A=L+F+R$.
Observe that for $y,z \in X$, the $(z,y)$-entry of $L$ equals $1$ if $\partial(z,y)=1$ and $\partial(x,z)= \partial(x,y)-1$, and $0$ otherwise. The $(z,y)$-entry of  $F$ is equal to $1$ if $\partial(z,y)=1$ and $\partial(x,z)= \partial(x,y)$, and $0$ otherwise. Similarly, the $(z,y)$-entry of $R$ equals $1$ if $\partial(z,y)=1$ and $\partial(x,z)= \partial(x,y)+1$, and $0$ otherwise. Consequently, for $v \in \e_i V \; (0 \le i \le \varepsilon)$ we have
\begin{equation}
	\label{eq:LRaction}
	L v \in \e_{i-1} V, \qquad  F v \in \e_{i} V, \qquad R v \in \e_{i+1} V.
\end{equation}
Observe also that $\Gamma$ is bipartite if and only if $F=0$.%

Recall that a \emph{$T$-module} is a subspace $W$ of $V$ such that $TW \subseteq W$. Let $W$ denote a $T$-module. Then, $W$ is said to be {\em irreducible} whenever $W$ is nonzero and $W$ contains no submodules other than $0$ and $W$. Since the algebra $T$ is semi-simple, any $T$-module is an orthogonal direct sum of irreducible $T$-modules.

Let $W$ denote an irreducible $T$-module. Observe that $W$ is an orthogonal direct sum of the non-vanishing subspaces $E_i^*W$ for $0 \leq i \leq \varepsilon$. The \emph{endpoint} of $W$ is defined as $r:=r(W)=\min \{i \mid 0 \le i\le \varepsilon, \; \e_i W \ne 0 \}$, and the \emph{diameter} of $W$ as $d:=d(W)=\left|\{i \mid 0 \le i\le \varepsilon, \; \e_i W \ne 0 \} \right|-1 $. It turns out that $\e_iW \neq 0$ if and only if $r \leq i \leq r+d$ $(0 \leq i \leq \varepsilon)$. The module $W$ is said to be \emph{thin} whenever $\dim(E^*_iW)\leq1$ for $0 \leq i \leq \varepsilon$. We say that two $T$-modules $W$ and $W^{\prime}$ are \emph{$T$-isomorphic} (or simply \emph{isomorphic}) whenever there exists a vector space isomorphism $\sigma: W \rightarrow W^{\prime}$ such that $\left( \sigma B - B\sigma \right) W=0$ for all $B \in T$. Note that isomorphic irreducible $T$-modules have the same endpoint and the same diameter. It turns out that $T\widehat{x}=\{B\widehat{x} \; | \; B \in T\}$ is the unique irreducible $T$-module with endpoint $0$. We refer to $T\widehat{x}$ as the {\em trivial $T$-module}.


\section{Graphs that support a uniform structure}
\label{sec:uniform}

In this section, we discuss the uniform property of bipartite graphs. The uniform property was first defined for graded partially ordered sets \cite{OldTer}. The definition was later extended to bipartite distance-regular graphs in \cite{MikTer} and then to an arbitrary bipartite graph in \cite{FMMM}. For ease of reference we recall the definition. Let $\Gamma=(X, \mathcal{R})$ denote a bipartite graph and let $V$ denote the standard module of $\Ga$. Fix $x\in X$, and let $\varepsilon=\varepsilon(x)$ denote the eccentricity of $x$. Let $T=T(x)$, $L$, and $R$ denote the corresponding Terwilliger algebra, lowering, and raising matrix, respectively.
\begin{definition}\label{3diag}
	 A \emph{parameter matrix} $U=(e_{ij})_{1\leq i,j\leq \varepsilon}$ is defined to be a tridiagonal matrix with entries in $\CC$, satisfying the following properties:
	\begin{enumerate}[label=(\roman*)]
		\item $e_{ii}=1$ $(1\leq i\leq \varepsilon)$,
		\item $e_{i, i-1}\neq0$ for $2 \le i \le \varepsilon$ or $e_{i-1,  i}\neq0$ for $2\leq i\leq \varepsilon$, and
		\item  the principal submatrix $(e_{ij})_{s\leq i, \, j\leq t}$ is nonsingular for $1\leq s\leq t\leq \varepsilon$.
	\end{enumerate}
	For convenience, we write $e^{-}_i:=e_{i, i-1}$ for $2\leq i\leq \varepsilon$ and $e^{+}_i:=e_{i, i+1}$ for $1\leq i\leq \varepsilon-1$. We also define $e^{-}_1:=0$ and $e^{+}_\varepsilon:=0$. 
\end{definition}

A \emph{uniform structure} of $\Gamma$ with respect to $x$ is a  pair $(U,f)$ where $f=\{f_i\}_{i=1}^\varepsilon$ is a vector in $\CC^\varepsilon$, such that
\begin{align}\label{uniformeq}
	e^{-}_iRL^2+LRL+ e^+_i L^2R=f_iL
\end{align}
 is satisfied on $E^*_iV$ for $1\leq i\leq \varepsilon$, where $\e_i\in T$ are the dual idempotents of $\Gamma$ with respect to $x$. If the vertex $x$ is clear from the context, we will simply use \emph{uniform structure} of $\Gamma$ instead of  \emph{uniform structure of $\Gamma$ with respect to $x$}.  The following result by Terwilliger plays an important role in the rest of this paper. 

\begin{theorem}[{\cite[Theorem 2.5]{OldTer}}]\label{oldpaper}
Let $\Gamma=(X, \mathcal{R})$ denote a bipartite graph and fix $x\in X$. Let $T=T(x)$ denote the corresponding Terwilliger algebra. Assume that $\Gamma$ admits a uniform structure with respect to $x$. Then, the following (i), (ii) hold:
\begin{enumerate}[label=(\roman*),]
\item Every irreducible $T$-module is thin. 
\item Let $W$ denote an irreducible $T$-module with endpoint $r$ and diameter $d$. Then, the isomorphism class of $W$ is determined by $r$ and $d$. 
\end{enumerate}
\end{theorem}

Assume now that $\Ga$ is a non-bipartite graph, and fix a vertex $x$ of $\Ga$. In what follows we define what it means for $\Gamma$ to support a uniform structure with respect to $x$ (see also \cite{FMMM}). To this end, we first define a certain graph $\Ga_f$.

\begin{definition}\label{def3.1}
	Let $\Gamma=(X,\mathcal{R})$ denote a non-bipartite graph. Fix $x \in X$ and define $\mathcal{R}_f = \mathcal{R} \setminus \{yz \mid \partial(x,y) = \partial(x,z)\}$. We define ${\Gamma_f}={\Gamma_f(x)}$ to be the graph with vertex set $X$ and edge set $\mathcal{R}_f$. We observe that $\Gamma_f$ is bipartite and connected. We say that $\Ga$ {\em supports a uniform structure with respect to} $x$, if $\Ga_f$ admits a uniform structure with respect to $x$. If the vertex $x$ is clear from the context, we will simply say that $\Ga$ \emph{supports a uniform structure} instead of \emph{supports a uniform structure with respect to $x$}. 
\end{definition}

With reference to Definition \ref{def3.1}, let $\varepsilon=\varepsilon(x)$ denote the eccentricity of $x$ and let $V$ denote the standard module for $\Ga$. Since the vertex set of $\Ga$ is equal to the vertex set of $\Ga_f$, observe that $V$ is also the standard module for $\Ga_f$. Let $T=T(x)$ be the Terwilliger algebra of $\Ga$. Recall that $T$ is generated by the adjacency matrix $A$ and the dual idempotents $\e_i$ ($0\leq i \leq \varepsilon$). Furthermore, we have $A=L+F+R$ where $L,F$, and $R$ are the corresponding lowering, flat, and raising matrices, respectively. Let $T_f=T_f(x)$ be the Terwilliger algebra of $\Ga_f$. As $\Ga_f$ is bipartite, the flat matrix of $\Ga_f$ with respect to $x$ is equal to the zero matrix. Moreover, the lowering and the raising matrices of $\Ga_f$ with respect to $x$ are equal to $L$ and $R$, respectively. It follows that for the adjacency matrix $A_f$ of $\Gamma_f$ we have $A_f=L+R$.  For $0 \le i \le \varepsilon$, note also that the $i$-th dual idempotent of $\Ga_f$ with respect to $x$ is equal to $\e_i$. Consequently, the algebra $T_f$ is generated by the matrices $L, R$, and $\e_i$ ($0\leq i \leq \varepsilon$). The following two results were proved in \cite{FMMM} and will be useful in the remaining sections of this paper. 

\begin{lemma}[{\cite[Lemma 3.4]{FMMM}}]\label{tf}
	With the above notation, let $W$ denote a $T$-module. Then,  the following (i), (ii) hold.
	\begin{enumerate}[label=(\roman*)]
		\item $W$ is a $T_f$-module.
		\item  If $W$ is a thin irreducible $T$-module, then $W$ is a thin irreducible $T_f$-module.
	\end{enumerate}
\end{lemma}

\begin{proposition}[{\cite[Proposition 3.6]{FMMM}}]\label{tilde}
With the above notation, let $W$ and $W'$ denote thin irreducible $T$-modules with endpoints $r$ and $r'$ and diameters $d$ and $d'$, respectively. Let $\{w_i\}_{i=0}^d$ be a basis for $W$ with $w_i \in \e_{r+i}W$. Similarly, let $\{w'_i\}_{i=0}^{d'}$ be a basis for $W'$ with $w'_i \in \e_{r'+i}W'$. Let $\beta_i \, (1\leq i \leq d)$ and $\gamma_i \, (0\leq i \leq d-1)$ be scalars such that $Lw_i=\beta_i w_{i-1}$ and $Rw_i=\gamma_i w_{i+1}$. Similarly, let $\beta'_i \, (1\leq i \leq d')$ and $\gamma'_i \, (0\leq i \leq d'-1)$ be scalars such that $Lw'_i=\beta'_i w'_{i-1}$ and $Rw'_i=\gamma'_i w'_{i+1}$. Then, $W$ and $W'$ are isomorphic as $T_f$-modules if and only if $r=r'$, $d=d'$, and $\beta_{i+1}/\beta'_{i+1}=\gamma'_i/\gamma_i$ for $0 \leq i \leq d-1$.
\end{proposition}


\section{Distance-regular graphs and strongly regular graphs}
\label{sec:drg}

In this section, we review some definitions and basic concepts regarding distance-regular graphs and strongly regular graphs. Let $\Gamma=(X, \R)$ denote a graph with diameter $D$. Recall that for $x,y \in X$ the \emph{distance} between $x$ and $y$, denoted by $\partial(x,y)$, is the length of a shortest $xy$-path. 
For an integer $i$, we define $\Gamma_i(x)$ by 
$$\Gamma_i(x)=\left\lbrace y \in X \mid \partial(x, y)=i\right\rbrace. 
$$
In particular,  $\Gamma(x)=\Gamma_1(x)$ is the set of neighbors of $x$. The subgraph of $\Ga$ induced on $\Ga(x)$ will be called the {\em local graph of} $\Gamma$ {\em with respect to} $x$ and denoted by $\Delta(x)$.

For an integer $k \geq 0$, we say that $\Gamma$ is \emph{regular} with valency $k$ whenever $|\Gamma(x)| = k$ for all $x\in X$. We say  that $\Gamma$ is \emph{distance-regular} whenever, for all integers $0 \leq h, i, j \leq D$ and all $x,y \in X$ with $\partial(x, y) = h$, the number $p^h_{ij}:= |\Gamma_{i}(x) \cap \Gamma_{j} (y)|$ is independent of the choice of $x,y$. The constants $p^h_{ij}$ are known as the \emph{intersection numbers} of $\Gamma$. For convenience, set $c_i := p^i _{1 \, i-1}  \; (1 \leq i \leq D)$, $a_i := p^i_{ 1i} \; (0 \leq i \leq D)$, $b_i := p^i_{ 1 \, i+1} \; (0 \leq i \leq D - 1)$,  $k_i := p^0_{ii} \; (0 \leq i \leq D)$, and $c_0 := 0, b_D := 0$.  Note that in a distance-regular graph, the eccentricity of every vertex is equal to $D$. 
For the rest of this section assume that $\Gamma$ is distance-regular. By the triangle inequality, for $0 \leq h, i, j \leq D$ we have $p^h_{ij}= 0$ (resp., $p^h_{ij}\neq 0$) whenever one of $h, i, j$ is greater than (resp., equal to) the sum of the other two. In particular, $c_i \neq 0 $ for $ 1 \leq i \leq D$ and $b_i \neq 0$ for $0 \leq i \leq D - 1$. Observe that $\Gamma$ is regular with valency $k = b_0 = k_1$ and $c_i +a_i +b_i = k $ for $ 0 \leq i \leq D$. Moreover, $\Ga$ is bipartite if and only if $a_i=0$ for $0 \le i \le D$.

Next, we recall the notion of near polygons, distance-regular graphs with classical parameters, and also strongly regular graphs. 

A distance-regular graph $\Ga$ is called a {\em near polygon} whenever $a_i=a_1 c_i$ for $1 \le i \le D-1$ and $\Ga$ does not contain $K_{1,1,2}$ as an induced subgraph \cite{No}. Here $K_{1,1,2}$ denotes the complete multipartite graph with three parts of size $1$, $1$, and $2$, respectively. 

The graph $\Ga$ is said to have {\it classical parameters} $(D,q,\alpha,\beta)$ 
whenever the intersection numbers of $\Ga$ satisfy
$$
  c_i = {i \brack 1} \Big( 1+\alpha {i-1 \brack 1} \Big) \qquad \qquad (1 \le i \le D), 
$$
$$
  b_i = \Big( {D \brack 1} - {i \brack 1} \Big) 
        \Big( \beta-\alpha {i \brack 1} \Big) \qquad \qquad (0 \le i \le D-1), 
$$
where
$$
  {j \brack 1} := 1+q+q^2+ \cdots + q^{j-1}. 
$$
In this case $q$ is an integer and $q \not \in \{0,-1\}$, see \cite[Proposition 6.2.1]{BCN}. 

 A graph $\Delta$ is \textit{strongly regular} with parameters $(n,k,\lambda,\mu)$ if it has vertex set of size $n$, it is regular of valency $k$, every two adjacent vertices have $\lambda$ common neighbors, and every two non-adjacent vertices have $\mu$ common neighbors. It is well-known that a connected regular graph is strongly regular if and only if it has exactly three different eigenvalues. Note that  if a strongly regular graph $\Delta$ is connected, then it is distance-regular with diameter $2$.  If $\Delta$ is not connected, then it is a disjoint union of cliques of the same size. We refer to \cite[Section 1.3]{BCN} for some well-known properties of strongly regular graphs.


\section{ Local eigenvalues - part I} 
\label{sec:locI}

Let $\Ga=(X,\mathcal{R})$ denote a non-bipartite distance-regular graph with diameter $D \ge 4$ and intersection numbers $b_i \, (0 \le i \le D-1)$, $c_i \, (1 \le i \le D)$. In this section, we summarize certain results of Go and Terwilliger \cite{GoTer, Ter} about thin irreducible $T$-modules with endpoint $1$. Fix $x \in X$ and recall the local graph $\Delta= \Delta(x)$.  Note that $\Delta$ has $k$ vertices and is regular with valency $a_1$.  The adjacency matrix $\tA$ of $\Delta$ is symmetric with real entries, and thus $\tA$ is diagonalizable with real eigenvalues.
We let $\eta_1, \eta_2 , \ldots , \eta_{k}$ denote the eigenvalues of $\tA$.  We call 
$\eta_1, \eta_2 , \ldots , \eta_{k}$ the {\em local eigenvalues of $\Gamma$ with respect to $x$}.  

\noindent Let $T=T(x)$ be the Terwilliger algebra of $\Ga$ and $\e_i \, (0 \le i \le D)$ be the dual idempotents of $\Ga$ with respect to $x$. We now consider the
first subconstituent $E^*_1V$, where $V$ is the standard module for $\Ga$. It is not hard to see that
the dimension of $E^*_1V$ is $k$ and $E^*_1V$ is invariant under the action of $E^*_1AE^*_1$.
We note that, for an appropriate ordering of the vertices of $\Gamma$, we have 
$$ E^*_1AE^*_1 = \left(\begin{array}{cc} \tA & 0 \\ 0 & 0 \end{array} \right).$$
Hence, the action of $E^*_1AE^*_1$ on $E^*_1V$ is essentially the adjacency map
for $\Delta$.
In particular, the action of $E^*_1AE^*_1$ on $E^*_1V$ is diagonalizable with eigenvalues $\eta_1, \eta_2, \ldots, \eta_{k}$.  We observe that the vector $s_{1} = \sum_{y \in \Ga(x)} \widehat{y}$ is contained in $E_{1}^{*}V$.  
One may easily show that $s_{1}$ is an eigenvector for $E^*_1AE^*_1$ 
with eigenvalue $a_1$.  Reordering the eigenvalues if necessary, we have $\eta_1 = a_1$.  For the rest of this paper, we assume that the local eigenvalues are ordered in this way.  
Now consider the orthogonal complement of $s_1$ in $E_1^*V$.  It is easy to see that this space is invariant under the action of $E^*_1AE^*_1$.  Thus, the restriction of the matrix $E^*_1AE^*_1$ to this space is diagonalizable with eigenvalues $\eta_{2}, \eta_{3},\ldots, \eta_{k}$.

\begin{definition}
	\label{TM}  \rm
	With the above notation, let $W$ denote a thin irreducible $T$-module with endpoint 1. Observe that $E^*_1W$ is a $1$-dimensional eigenspace for $E^*_1AE^*_1$, and let $\eta $ denote the corresponding eigenvalue. Since $E^*_1W$ is contained in $E^*_1V$, the eigenvalue $\eta $ is one of $\eta_{2}, \eta_{3},\ldots, \eta_{k}$. We refer to $\eta $ as the {\em local eigenvalue} of $W$.
\end{definition}

\begin{theorem}[{\cite[Theorem 12.1]{Ter}}]\label{iso}
	With the above notation, let $W$ denote a thin irreducible $T$-module with endpoint $1$ and local eigenvalue $\eta$, and $W'$ denote an irreducible $T$-module. Then, $W$ and $W'$ are isomorphic as $T$-modules if and only if  $W'$ is thin with endpoint $1$ and local eigenvalue $\eta$. Consequently, the isomorphism class of  $W$ is determined by $\eta$.
\end{theorem}

To further describe local eigenvalues of irreducible $T$-modules with endpoint $1$, we introduce the following notation: for $z\in \mathbb{R}$, we define 
$$ 
\tilde{z}=\begin{cases}
	-1-\frac{b_1}{1+z} &\hbox{ if $z\neq-1$},\\
	\infty &\hbox{ if $z=-1$}.
\end{cases}
$$

Let $\theta_0>\theta_1>\ldots>\theta_D$ denote the eigenvalues of $\Ga$. It follows from \cite[Lemma 2.6]{JKT} that $\theta_1,\theta_D$ are not equal to $ -1$, also $\tilde{\theta_1}<-1$ and  $\tilde{\theta_D}\ge 0$. Let $W$ denote a thin irreducible $T$-module with endpoint 1 and local eigenvalue $\eta$.  By \cite[Theorem 1]{Ter1}, we have that $\tilde{\theta_1}\le \eta \le \tilde{\theta_D}$.

\begin{proposition}[{\cite[Corollary 7.8]{inequ}}]
	\label{prop:loc=loc}
 With the above notation, suppose every irreducible $T$-module with endpoint $1$ is thin. Let $\Phi$  denote the set of distinct scalars among $\eta_2, \eta_3, \ldots, \eta_k$.  Then, for every  $\eta \in \Phi$ there exists a  thin irreducible $T$-module with endpoint $1$  and local eigenvalue $\eta$. Up to isomorphism, there are no further irreducible $T$-modules with endpoint $1$.
\end{proposition}

\begin{proposition}[{\cite[Theorem 10.3]{GoTer}, \cite[Theorem 11.4]{Ter}}]\label{prop:diam}
	With the above notation,  let $W$ denote a thin irreducible $T$-module with endpoint $1$, diameter $d$, and local eigenvalue $\eta$. Then, the following (i), (ii) hold. 
	\begin{enumerate}[label=(\roman*)]
		\item If $\eta \in \{\tilde{\theta_1}, \tilde{\theta_D}\}$, then $d=D-2$. 
		\item If $\tilde{\theta_1}<\eta< \tilde{\theta_D}$, then $d=D-1$. 
	\end{enumerate}
 \end{proposition}

In what follows we will be using the following notation.

\begin{notation}\label{not3.2}
	Let $\Ga=(X,\mathcal{R})$ denote a non-bipartite distance-regular graph with classical parameters $(D,q,\alpha,\beta)$ with $D \ge 4$, $q \ge 2$, intersection numbers $b_i \, (0 \le i \le D-1)$, $c_i \, (1 \le i \le D)$,  and eigenvalues $\theta_0>\theta_1>\ldots>\theta_D$. Let $V$ denote the standard module of $\Ga$. Fix $x\in X$, and let $\Delta=\Delta(x)$ denote the local graph of $\Gamma$ with respect to $x$. Let $T=T(x)$ be the Terwilliger algebra of $\Ga$ and $\e_i \, (0 \le i \le D)$ be the dual idempotents of $\Ga$ with respect to $x$. Let $L$, $F$, and $R$ denote the corresponding lowering, flat, and raising matrix, respectively. Let ${\Gamma_f}={\Gamma_f(x)}$ be as defined in Definition~\ref{def3.1} and $T_f=T_f(x)$ be the Terwilliger algebra of $\Ga_f$. Recall  that $T_f$ is generated by the matrices $L, R$, and $\e_i$ ($0\leq i \leq D$). We assume that every irreducible $T$-module with endpoint $1$ is thin and that $\Ga$ supports a uniform structure with respect to $x$.
\end{notation}


\section{Local eigenvalues - part II}
\label{sec:locII}

With reference to Notation \ref{not3.2}, in this section, we describe local eigenvalues of irreducible $T$-modules with endpoint $1$. 

Let $W$ denote an irreducible $T$-module with endpoint $1$ (recall that $W$ is thin by assumption), and let $\eta$ denote the  local eigenvalue of $W$. According to Theorem \ref{iso}, the isomorphism class of $W$ is determined by $\eta$. Applying \cite[Corollary 4.12]{TerpartII}, we find that

\begin{equation}\label{4eta}
\eta \in \left\lbrace -q -1, \, \, \beta-\alpha - 1, \, \, -1, \, \, \alpha q\frac{q^{D-1}-1}{q-1}-1\right\rbrace.
\end{equation}
Let us refer to these values in \eqref{4eta} as $\eta_1$, $\eta_2$, $\eta_3$, and $\eta_4$, respectively. Note that $\eta_1$, $\eta_2$, and $\eta_3$ are distinct, $\eta_4\ne\eta_1$, $\eta_4=\eta_2$ if and only if $\beta=\alpha\frac{q^{D}-1}{q-1}$, and $\eta_4=\eta_3$ if and only if $\alpha=0$ (see \cite[Section 13]{Ter} for more details). As we proceed, when we mention $\eta_4$ we implicitly assume that $\alpha \ne 0$ and $\beta \ne \alpha\frac{q^{D}-1}{q-1}$. We have that the diameter $d$ of $W$ is $D-2$ if  $\eta\in \{\eta_1,  \,  \eta_2\}$, and $D-1$ if $\eta\in \{\eta_3,  \,  \eta_4\}$, see \cite[Section 13]{Ter}.
This implies that we have at most four thin irreducible $T$-modules with endpoint $1$, up to isomorphism. 

\begin{lemma}
	\label{lem:noniso}
	With reference to Notation~\ref{not3.2}, assume $\alpha \ne 0$. Let $W$ and $W'$ denote irreducible $T$-modules with endpoint $1$. If $W$ and $W'$ are nonisomorphic as $T$-modules, then they are also nonisomorphic as $T_f$-modules. 
\end{lemma}
\begin{proof}
If $W$ and $W'$ have different diameters, then they are clearly nonisomorphic as $T_f$-modules. Therefore, assume that $W$ and $W'$ have the same diameter $d$. Suppose first that $d=D-2$. Then, without loss of generality, we assume that the local eigenvalue of $W$ is $\eta_1$ and the local eigenvalue of $W'$ is $\eta_2$. Let $v$ ($v'$, respectively)  denote a non-zero vector in $\e_1W$ ($\e_1W'$, respectively). Then, $\{\e_{i+1} A_i v\}_{i=0}^{D-2}$ ($\{\e_{i+1} A_i v'\}_{i=0}^{D-2}$) is a basis for $W$ ($W'$, respectively), see \cite[Theorem 10.3]{GoTer}. Let $\beta_i, \beta'_i \; (1 \le i \le D-2)$ and $\gamma_i, \gamma'_i \; (0 \le i \le D-3)$ denote the corresponding scalars as in Proposition \ref{tilde}. Using \cite[Theorem 10.6]{GoTer} and \cite[Section 13, p. 184]{Ter}, we find that $\gamma_i = \gamma'_i = c_{i+1}$, $\beta_i=b_{i+1}$, and  
$$
\beta'_i= \frac{c_{i+1}}{c_i} \frac{q^D-q^{i+1}}{q^{i+1}-1} \frac{q^i-1}{q-1} \Bigg( \beta - \alpha \frac{q^i-1}{q-1} \Bigg).
$$
By Proposition \ref{tilde}, $W$ and $W'$ are isomorphic as $T_f$-modules if and only if $\beta_i = \beta'_i$ for $1 \le i \le D-2$. It is now easy to see that $\beta_1 = \beta'_1$ implies $\alpha=0$ or $q + \beta-\alpha=0$. As $\alpha \ne 0$ by assumption, we must have $q + \beta-\alpha=0$. However, one could now easily check that this implies $b_1 < 0$, a contradiction. This shows that $\beta_1 \ne \beta'_1$ and consequently $W$ and $W'$ are nonisomorphic as $T_f$-modules.

Suppose now that $d=D-1$. We proceed similarly as above. Without loss of generality, assume that the local eigenvalue of $W$ is $\eta_3$ and the local eigenvalue of $W'$ is $\eta_4$. Let $v$ ($v'$, respectively) denote a nonzero vector in $\e_1W$ ($\e_1W'$, respectively). Then, $\{\e_{i+1} A_i v\}_{i=0}^{D-1}$ ($\{\e_{i+1} A_i v'\}_{i=0}^{D-1}$) is a basis for $W$ ($W'$, respectively), see \cite[Theorem 11.4]{Ter}. Let $\beta_i, \beta'_i \; (1 \le i \le D-2)$ and $\gamma_i, \gamma'_i \; (0 \le i \le D-3)$ denote the corresponding scalars as in Proposition \ref{tilde}. Using \cite[Theorem 11.6]{Ter} and \cite[Section 13, p. 184]{Ter}, we find that $\gamma_i = \gamma'_i = c_{i+1}$, $\beta_i=b_i$ and  
$$
\beta'_i= \frac{c_{i+1}}{c_i} \frac{q^D-q^i}{q^{i+1}-1} \frac{q^i-1}{q-1} \Bigg( \beta - \alpha \frac{q^{i+1}-1}{q-1} \Bigg).
$$
By Proposition \ref{tilde}, $W$ and $W'$ are isomorphic as $T_f$-modules if and only if $\beta_i = \beta'_i$ for $1 \le i \le D-1$. It is now easy to see that $\beta_1 = \beta'_1$ implies $\alpha=0$ or $\alpha q + q - \beta +\alpha=0$. As $\alpha \ne 0$ by assumption, we must have that  $\alpha q + q - \beta +\alpha=0$. However, one could now easily check that this implies $\beta_2 \ne \beta'_2$, and consequently $W$ and $W'$ are nonisomorphic as $T_f$-modules.
\end{proof}

\begin{corollary}
	\label{cor:local}
	With reference to Notation~\ref{not3.2}, assume $\alpha \ne 0$. Then, the following (i), (ii) hold.
	\begin{itemize}
		\item[(i)] $\Gamma$ has, up to isomorphism, exactly two irreducible $T$-modules with endpoint $1$; one	of these modules has diameter $D-2$ and the other one has diameter $D-1$.
		\item[(ii)] $\Delta$ is a non-complete strongly regular graph.
	\end{itemize}
\end{corollary}
\begin{proof}
	(i) Recall that by assumption $\Gamma$ supports a uniform structure with respect to $x$. In virtue of Theorem~\ref{oldpaper}, the isomorphism class of any irreducible $T$-module with endpoint $1$ is determined by its diameter. By taking into account Proposition~\ref{prop:diam} and Lemma \ref{lem:noniso}, this implies that $\Gamma$ must admit, up to isomorphism, at most two irreducible $T$-modules with endpoint $1$. Assume for a moment that $\Ga$ has, up to isomorphism, just one irreducible $T$-module with endpoint $1$. Then, $\Ga$ is bipartite or almost bipartite, see \cite[Theorem 1.3]{CN}. Consequently, $a_1=a_2=0$, which implies that $\alpha=0$ and $\beta=1$. It follows that $\Ga$ is bipartite, a contradiction. Therefore, $\Gamma$ has, up to isomorphism, exactly two irreducible $T$-modules with endpoint $1$. By Proposition~\ref{prop:diam} and Lemma \ref{lem:noniso}, one	of these modules has diameter $D-2$, and the other has diameter $D-1$. 
	
	\smallskip \noindent
	(ii) By Proposition \ref{prop:loc=loc} and (i) above, the local graph $\Delta$ of $\Gamma$ has at most three different eigenvalues. Recall that $\Delta$ is regular with valency $a_1$. If $\Delta$ has two different eigenvalues, then $\Delta$ is a disjoint union of cliques (of the same size). Note that $\Delta$ is not complete (otherwise $\Ga$ is complete, too). If $\Delta$ has three different eigenvalues, then the valency $a_1$ is an eigenvalue with multiplicity $1$. It follows that $\Delta$ is connected, and so it is a connected strongly regular graph. 
\end{proof}

Next, we describe the eigenvalues of the local graph $\Delta$. 

\begin{theorem}
	\label{lem:local1}
	With reference to Notation~\ref{not3.2}, assume $\alpha \ne 0$. Let $(n=b_0, k=a_1, \lambda, \mu)$ denote the parameters of $\Delta$. Then, the following (i), (ii) hold. 
	\begin{itemize}
		\item[(i) ] $\Delta$ is a connected strongly regular graph with  eigenvalues
		$$
		a_1, \; \alpha q \frac{q^{D-1}-1}{q-1}-1, \; -q-1.
		$$ 
		
		\item [(ii)] We have 
		$$
		\beta=\alpha\frac{q^{D+1}-1}{q-1}-q.
		$$
		In particular,
		$$
		  n = b_0 = \frac{(q^D-1)(\alpha q^{D+1}-q^2+q-\alpha)}{(q-1)^2}, \quad k=a_1=\frac{(q+1)(\alpha q^D - q - \alpha +1)}{q-1},
		$$
		$$
		  \lambda =  \frac{\alpha q^D + \alpha q^2 - q^2 - \alpha q - q - \alpha + 2}{q-1}, \quad \mu=\alpha(q+1).
		$$
	\end{itemize}
\end{theorem}
\begin{proof}
	(i) Assume first that $\Delta$ is a disjoint union of cliques. Then, it is clear that its eigenvalues are its valency $a_1$ and $-1$. As $\Delta$ is not complete, it is disconnected, and so $a_1$ is an eigenvalue of $\Delta$ with multiplicity greater or equal to two. Consequently, $\Ga$ has an irreducible $T$-module $W$ with endpoint $1$ and local eigenvalue $a_1$. As $a_1$ is non-negative, we have either $a_1 = \eta_2$, or $a_1 = \eta_4$. If $a_1=\eta_2$, then $\alpha=0$, a contradiction. Consequently, $a_1 = \eta_4$, which implies that the diameter of $W$ is $D-1$. Let $W'$ denote an irreducible $T$-module with endpoint $1$ and local eigenvalue $-1=\eta_3$. As $\eta_1, \eta_2, \eta_3$ are pairwise distinct, the discussion at the beginning of this section implies that the diameter of $W'$ is $D-1$, contradicting Corollary \ref{cor:local}(i). This shows that $\Delta$ is a connected strongly regular graph. 
	
	Let $r > 0$ and $s < 0$ denote the eigenvalues of $\Delta$, different from $a_1$. As $\Delta$ is not a disjoint  union of cliques, we have that $s \ne -1$. As one of the two non-isomorphic $T$-modules with endpoint $1$ has diameter $D-1$, it follows that one of $r$ or $s$ must be equal to  $\eta_4$. Therefore, as $r$ is positive, we have the following three possibilities:
	$$
	  (r, s) \in \{(\eta_4,\eta_1), (\eta_4, \eta_2), (\eta_2, \eta_4)\}.
	$$
	
	Assume first that $(r, s) \in \{ (\eta_4, \eta_2), (\eta_2, \eta_4) \}$. Using \cite[Theorem 1.3.1 (iii)]{BCN}, we find 
	$$
	  \mu = k+ rs = \frac{\alpha (\beta  - \alpha) (q^D -  q)}{q-1} + 1, \quad
	  \lambda = \mu + r + s = \frac{\alpha (\beta - \alpha + 1)(q^D-q) }{q-1} + \beta-2,
	$$
	$$
	  n = 1 + k +\frac{k(k-\lambda-1)}{\mu} = \frac{\beta (q^D-1)}{(\beta - \alpha)(q^D  - q)+q-1}.
	$$
	Setting $n=b_0$, we find either $\beta=0$ or $\beta=\alpha$. If $\beta=0$, then $b_0 = 0$, a contradiction. If $\beta=\alpha$, then $b_1=0$, a contradiction. This shows that $(r, s)=(\eta_4, \eta_1)$.
	
	\smallskip \noindent
	(ii) Recall that $\beta \ne 0$; otherwise $b_0=0$, a contradiction. Computing $n$ similarly as above and setting $n=b_0$, we get  $\beta=\alpha\frac{q^{D+1}-1}{q-1}-q$. The second part of the claim now follows from \cite[Theorem 1.3.1 (iii)]{BCN}.
\end{proof}


\section{The local graph $\Delta$}
	\label{sec:localgraph}

With reference to Notation~\ref{not3.2}, in this section, we continue to study our local graph $\Delta$. To this end, we recall a result by Neumaier about connected strongly regular graphs that we will find useful later in this section. 

\begin{theorem}[{\cite[Theorem 4.7]{AN}}]\label{neumaier}
	Let $G$ be a strongly regular graph with parameters $(n, k,
	\lambda, \mu)$ and eigenvalues $k >	r > s$. Assume that $s < -1$ is integral. Then, at least one of the following conditions must hold:
	\begin{enumerate}[label=(\roman*)]
		\item $r \leq \frac{s(s+1)(\mu+1)}{2} -1$;
		\item $\mu= s^2$ (in which case $G$ is a Steiner graph derived from a Steiner $2$-system in which each line contains $s$ points);
		\item $\mu = s(s+1)$ (in which case $G$ is a Latin square graph derived from an $s$-net).
	\end{enumerate}
\end{theorem}

\begin{lemma}\label{claim1}
	With reference to Notation~\ref{not3.2}, assume $\alpha \ne 0$. Then, the local graph $\Delta$ is not a conference graph. 
\end{lemma}

\begin{proof}
	Suppose to the contrary that $\Delta$ is a conference graph. Then, $\lambda=\mu-1$ and $k=2 \mu$ by \cite[Section 1.3]{BCN}. So, by Theorem~\ref{lem:local1}(ii), we have that $\alpha=\frac{q^2-1}{q^D-q}$ and 
	$$
	  0 = k-2 \mu = \frac{(q+1)(q^{D+1} - 2 q^2 + 1)}{q^D-q},
	$$
	a contradiction since $q \ge 2$ and $D \ge 4$. 
\end{proof}

\begin{lemma}\label{claim2}
		With reference to Notation~\ref{not3.2}, assume $\alpha \ne 0$. Then, $\eta_4\geq 1$. 
\end{lemma}
\begin{proof}
	Recall that $\eta_4$ is a non-negative eigenvalue of $\Delta$ and that $\Delta$ is not a conference graph. Therefore, it suffices to show that $\eta_4 \ne 0$. If $\eta_4=0$ then $\alpha=\frac{q-1}{q^D-q}$, and so we get $\mu=\frac{q^2-1}{q^D-q}$ by Theorem~\ref{lem:local1}(ii) . As $\mu\geq1$, $q \ge 2$ and $D \ge 4$, this implies that 
	$$
	  q^3  < q^{D-1} + q^{D-1}-q \le q^D-q \le q^2-1,
	$$ 
	a contradiction. This shows that  $\eta_4\geq 1$. 
\end{proof}

\begin{theorem}
	\label{thm:alpha}
	With reference to Notation~\ref{not3.2}, assume $\alpha \ne 0$.  Then, $\alpha\in\left\lbrace q, q+1\right\rbrace $. 
\end{theorem}
\begin{proof}
	  To prove the statement, it is enough to show that the case (i) in Theorem~\ref{neumaier} can never be satisfied. Suppose by contradiction that the case (i) in Theorem~\ref{neumaier} holds. Then,
	  \begin{equation}
	  	\label{eq:alpha}
	  	\alpha(2q^{D-1}-q^3-q^2+q-1) \le q^2-1
	  \end{equation}
	  is satisfied, and this implies that
	  $$
	    \alpha \le \frac{q^2-1}{2q^{D-1}-q^3-q^2+q-1},
	  $$
	  as $q \ge 2$ and $D \ge 4$. Recall that $\mu=\alpha(q+1) \ge 1$, and so $\alpha \ge 1/(q+1)$. This, together with the above inequality, implies that 
	  $$
	    2q^{D-1}-q^3-q^2+q-1 \le (q+1)(q^2-1)=q^3 + q^2-q-1,
	  $$
	  or 
	  $$
	     q^{D-1} \le q^3+q^2-q.
	  $$
	  This yields $D=4$, and so \eqref{eq:alpha} gives 
	  $$
	    \alpha \le \frac{q+1}{q^2+1}.
	  $$
	  Now $\mu=\alpha(q+1)$ yields $\mu=1$, and so $\alpha = 1/(q+1)$. It follows that $\lambda=q^2-q-\frac{1}{q+1}$, contradicting the integrality of $\lambda$. It follows from Theorem~\ref{neumaier} that either $\mu=\alpha(q+1) = s^2 = (q+1)^2$ (in which case $\alpha=q+1$), or $\mu=\alpha(q+1) = s(s+1) = q(q+1)$ (in which case $\alpha=q$).
\end{proof}

\begin{corollary}
	\label{xc}
	With reference to Notation~\ref{not3.2}, assume $\alpha \ne 0$.  Then,
	$$
	  (D,q,\alpha,\beta) \in \Bigg\{\Bigg(D, q, q+1, \frac{q^{D+1} (q+1)-q^2-1}{q-1} \Bigg), \;   \Bigg(D, q, q, \frac{q^2 (q^D-1)}{q-1}\Bigg) \Bigg\}.
	$$
\end{corollary}
\begin{proof}
	Immediately follows from Theorems \ref{lem:local1} and \ref{thm:alpha}.
\end{proof}


\section{The main result}
\label{sec:main}

With reference to Notation~\ref{not3.2}, assume $\alpha \ne 0$.  Recall that by Corollary \ref{xc} the classical parameters of $\Ga$ belong to one of the following two families:
$$
\Bigg(D, q, q+1, \frac{q^{D+1} (q+1)-q^2-1}{q-1} \Bigg), \qquad   \Bigg(D, q, q, \frac{q^2 (q^D-1)}{q-1}\Bigg).
$$
In this section, we show that the graphs with classical parameters belonging to the first of the above two families do not exist. We will also show that if $\Ga$ has classical parameters belonging to the second of the above two families, then $D \equiv 0 \pmod{6}$. The proofs are somehow technical (the main idea is to prove that certain expressions are not integers), but to make things easier for the reader we decided to include most of the details. For ease of reference, we will continue to refer to Notation~\ref{not3.2} in the rest of this section, however we must mention that the proofs of Theorems \ref{fam1} and \ref{fam2}  (as well as Remark \ref{rem}) remain valid even if we omit our assumptions that every $T$-module with endpoint $1$ is thin, and that $\Ga$ supports a uniform structure with respect to $x$.

\begin{theorem}\label{fam1}
	With reference to Notation~\ref{not3.2}, the family of distance-regular graphs with classical parameters 
	$$
	  (D,q,\alpha,\beta)=\Big(D, q, q+1, \frac{q^{D+1} (q+1)-q^2-1}{q-1} \Big)
	$$
does not exist.
\end{theorem}

\begin{proof}
Assume first that $D\geq 6$. By \cite[Lemma 4.1.7]{BCN}, we compute the intersection number 
\begin{equation}\label{p633}
p^{6}_{3 3}=\frac{(q+1)(q^2+1) (q^2-q+1) (q^3+2 q^2+2q+2) A_1 A_2 A_3}{(q+2) (q^2+2q+2)},
\end{equation}
where 
$$
 A_1 = q^4+q^3+q^2+q+1, \quad  A_2=q^4+2q^3+2q^2+2q+2, \quad A_3= q^5+2q^4+2q^3+2q^2+2q+2.
$$
It follows that $q^2+2q+2$ divides $(q^2+1) ( q^3+1) (q^3+2 q^2+2q+2) A_1 A_2 A_3$. It is easy to see that
\begin{align*}
	q^2+1 &= (q^2+2q+2)-(2q+1) \\
	q^3+1 &= (q-2) (q^2+2q+2) + 2q+5 \\
	q^3+2q^2+2q+2 &= q(q^2+2q+2)+2 \\
	q^4+q^3+q^2+q+1 &= (q^2 - q + 1)(q^2+2q+2) + q-1\\
	q^4+2q^3+2q^2+2q+2 &=  q^2(q^2+2q+2) + 2(q+1)\\
	q^5+2q^4+2q^3+2q^2+2q+2 &= (q^3+2)(q^2+2q+2) - 2(q+1),
\end{align*}
and so it follows that $q^2+2q+2$ must divide $8 (q-1) (q+1)^2 (2 q+1) (2 q+5)= 8 (4 q^3 + 8 q^2 - 11 q  - 5)(q^2+2q+2)+40(3q +1)$. This further implies that  $q^2+2q+2$ must divide $40(3q +1)$. It is now straightforward to check that the only integers $q \ge 2$, such that $q^2+2q+2$  divides $40(3q +1)$ and the expression \eqref{p633} is an integer, are $q=2$ and $q=4$. Assume for the moment that $q\in\{2,4\}$. If $D\geq8$, using \cite[Lemma 4.1.7]{BCN} again,  it is easy to see that the intersection number $p^8_{4 4}$ is not an integer, implying that $D\in\{6,7\}$. Let $f_2$ denote the multiplicity of the eigenvalue $\theta_2$. Using \cite[Lemma 8.4.3]{BCN}, we easily find that 
$$
f_2 = \frac{(q^D-1)(q^D-q)(q^{D+1}+q^D+2)(q^{D+2}+q^{D+1}-q^2-1)(q^{2D+1}+q^{2D}-q^{D+1}+q^D-2q^3)}
                     {(q-1)^2(q+1)^2(q^{D+1}+q^D-3q+1)(q^{D+1}+q^D-2q^2)},
$$ 
and so $f_2$ is not an integer for $q \in \{2,4\}$ and $D \in \{6,7\}$. This shows that $D \in \{4,5\}$.

Assume now that $D=4$. Note that in this case we have
$$
  f_2 = \frac{q^2(q^2+1)(q^2+q+1)^2(q^3+q^2+1)(q^5+q^4+2)(q^5+2q^4+2q^3+2q^2+q+2)}{(q+1)(q^2+2q+2)(q^4+2q^3+2q^2+2q-1)},
$$
and so $(q^2+2q+2)$ must divide $60(3q+4)$. It is now easy to check that the only integer for which $(q^2+2q+2)$ divides $60(3q+4)$ is $q=2$, but in this case $f_2$ is not an integer. 

Assume finally that $D=5$. Then, we have that 
$$
  k_2= \frac{q^3(q^2+1)(q^4+q^3+q^2+q+1)(q^4+2q^3+2q^2+2q+2)(q^6+2q^5+2q^4+2q^3+2q^2+q+1)}{q+2},
$$
and so $(q+2)$ must divide $60720$.  It is now easy to check that there are no integers $q \ge 2$ for which $f_2$ is an integer and $(q+2)$ divides $60720$, simultaneously. This completes the proof.
\end{proof}

\begin{theorem}\label{fam2}
With reference to Notation~\ref{not3.2}, assume that $D \not \equiv 0 \pmod{6}$. Then, the family of distance-regular graphs with classical parameters 
$$ 
  (D,q,\alpha,\beta)=\Big(D, q, q, \frac{q^2 (q^D-1)}{q-1} \Big).
$$
does not exist.
\end{theorem}
\begin{proof}
We first claim that $D$ is even. To prove the claim, we compute the multiplicity $f_2$ corresponding to the eigenvalue $\theta_2$. Using \cite[Lemma 8.4.3]{BCN}, we get
 \begin{equation*}\label{f2}
f_2=\frac{q^2(q^D-1) (q^{D+1}+1)(q^{2D-2}-q^{D-2}+q^{D-3}-1)}{(q-1)^2(q+1)},
 \end{equation*}
and so $(q+1)$ must divide the numerator of the above expression. Assume that $D$ is odd. Then, we have
\begin{align*}
	q^2 &=(q-1)(q+1) + 1 \\
	q^D-1 &= (q^{D-1} - q^{D-2} + q^{D-3} - \cdots -q+1)(q+1)-2 \\
	q^{D+1}+1 &= (q^D - q^{D-1} + q^{D-2}  - \cdots +q-1)(q+1) + 2 \\
	q^{2D-2}-q^{D-2}+q^{D-3}-1 &= (A - 2q^{D-3} + 3B)(q+1) + 2,
\end{align*}
where $A=q^{2D-3}-q^{2D-4} + \cdots -q^{D-1} + q^{D-2}$ and $B=q^{D-4}-q^{D-5} + \cdots +q-1$.
This  implies that $(q+1)$ divides $8$, and so $q \in \{3,7\}$. Now if $q=3$, then 
$$
  f_2 = \frac{(3^D-1)(3^{D+1}+1)(3^{2D+1}-2 \cdot 3^D-27)}{48}.
$$
Since $D$ is odd, we have that $(3^D-1)$, $(3^{D+1}+1)$, and $(3^{2D+1}-2 \cdot 3^D-27)$ are all congruent to $2$ modulo $4$, and so $f_2$ is not an integer. Similarly, if $q=7$, then 
$$
f_2 = \frac{(7^D-1)(7^{D+1}+1)(7^{2D+1}-6 \cdot 7^D-343)}{2016}.
$$
Again, since $D$ is odd, we have that $(7^D-1)$, $(7^{D+1}+1)$, and $(7^{2D+1}-6 \cdot 7^D-343)$ are all congruent to $2$ modulo $4$, and so $f_2$ is not an integer. This shows the claim, and so $D$ is even.

In the rest of the proof, we use similar arguments as above, and so we omit the details. Using \cite[Lemma 8.4.3]{BCN} again, we get that the multiplicity $f_3$ corresponding to the eigenvalue $\theta_3$ is equal to 
$$
  f_3 = \frac{(q^D-1)(q^{D+1}+1)(q^{2D+1}-q^{D+1}+q^D-q)(q^{2D-2}-q^{D-2}+q^{D-3}-q^2)}
  {(q-1)^3(q+1)(q^2+q+1)},
$$
which implies that $(q^2+q+1)$ divides the numerator of the above expression. Assume first that $D \equiv 2 \pmod{6}$. Then, it is easy to check that 
\begin{align*}
	q^D-1 &= g_1(q) (q^2+q+1) - (q+2) \\
	q^{D+1}+1 &= g_2(q) (q^2+q+1) + 2 \\
	q^{2D+1}-q^{D+1}+q^D-q &= g_3(q) (q^2+q+1) - 3(q+1) \\
	q^{2D-2}-q^{D-2}+q^{D-3}-q^2&= g_4(q) (q^2+q+1) - (q+2) 
\end{align*}
for certain polynomials $g_1, g_2, g_3$, and $g_4$ with integer coefficients. It follows that $(q^2+q+1)$ divides $6(q+2)^2(q+1)$. As $6(q+2)^2(q+1) = 6(q+4)(q^2+q+1) + 18 q$, this further implies that $(q^2+q+1)$ divides $18 q$. One could now easily check that this is never true for $q \ge 2$. 

Assume finally that $D \equiv 4 \pmod{6}$. Then, it is again easy to check that 
\begin{align*}
	q^D-1 &= h_1(q) (q^2+q+1) + (q-1 )\\
	q^{D+1}+1 &= h_2(q) (q^2+q+1) - q \\
	q^{2D+1}-q^{D+1}+q^D-q &= h_3(q) (q^2+q+1) + (q+2) \\
	q^{2D-2}-q^{D-2}+q^{D-3}-q^2&= h_4(q) (q^2+q+1) + 3(q+1) 
\end{align*}
for certain polynomials $h_1, h_2, h_3$, and $h_4$ with integer coefficients. It follows that $(q^2+q+1)$ divides $3(q-1)q(q+2)(q+1)$. As $3(q-1)q(q+2)(q+1) = 3(q^2+q-3)(q^2+q+1) + 9$, this further implies that $(q^2+q+1)$ divides $9$, a contradiction. This finishes the proof.
\end{proof}

We are now ready to state our main result.

\begin{theorem}
With reference to Notation~\ref{not3.2}, we have that either $\alpha=0$, or $D \equiv 0 \pmod{6}$ and $\Ga$ has classical parameters
$$\Big(D, q, q, \frac{q^2 (q^D-1)}{q-1} \Big).$$
\end{theorem}

\begin{proof}
The result immediately follows from Corollary \ref{xc}, Theorem \ref{fam1}, and Theorem \ref{fam2}.
\end{proof}

We finish the paper with a remark and a conjecture. 
\begin{remark}
	\label{rem}
	\rm
With reference to Notation~\ref{not3.2}, assume that  $D \equiv 0 \pmod{6}$ and $\Ga$ has classical parameters
$$
  \Big(D, q, q, \frac{q^2 (q^D-1)}{q-1} \Big).
$$
Then, using \cite[p. 127, formulae (1c)]{BCN} and \cite[Lemma 8.4.3]{BCN}, we find that the valency $k_D$ and the multiplicity $f_D$ of the eigenvalue $\theta_D$ are given by
$$
k_D = q^{\frac{D(D+1)}{2}+1}\prod_{i=1}^{D-1}{\left(q \frac{q^D - 1}{q^{i}- 1} - 1\right)},
$$
$$
f_D = (q^D (q + 1) - q) \prod_{i=2}^{D}{\left(q^{i + 1}\frac{q^D - 1}{q^i - 1} + 1\right)}.
$$
With computer assistance, we obtain that $k_D$ and $f_D$ are never integers for $q,D\le3000$.
\end{remark}

The above remark thus encourages us to put forward the following conjecture.

\begin{conjecture}
There exists no distance-regular graph with classical parameters 
$$
  \Bigg(D, q, q, \frac{q^2 (q^D-1)}{q-1} \Bigg)
$$
with $q \ge 2$ and $D \ge 4$. 
\end{conjecture}

With reference to Notation~\ref{not3.2}, note that if the above conjecture is true, then we must have $\alpha=0$.

\begin{remark}
	\rm
	We would like to point out that the dual polar graphs satisfy assumptions of Notation \ref{not3.2} with $\alpha=0$. We will study this case in detail in our future paper.
\end{remark}

\subsection*{Acknowledgment}
Blas Fernández's work is supported in part by the Slovenian Research Agency (research program P1-0285, research projects J1-2451, J1-3001 and J1-4008). \v{S}tefko  Miklavi\v{c}'s  research is supported in part by the Slovenian Research Agency (research program P1-0285 and research projects J1-1695, N1-0140, N1-0159, J1-2451, N1-0208, J1-3001, J1-3003, J1-4008 and J1-4084).
Roghayeh Maleki and Giusy Monzillo's research are supported in part by the Ministry of Education, Science and Sport of Republic of Slovenia (University of Primorska Developmental funding pillar).



\begin{thebibliography}{10}
	
	\bibitem{BCN}
	Andries Brouwer, Arjeh Cohen, and Arnold Neumaier.
	\newblock Distance-regular graphs. 1989.
	\newblock {\em Ergeb. Math. Grenzgeb.(3)}, 1989.
	
	\bibitem{CN}
	Brian Curtin and Kazumasa Nomura.
	\newblock 1-homogeneous, pseudo-1-homogeneous, and 1-thin distance-regular
	graphs.
	\newblock {\em J. Combin. Theory Ser. B}, 93(2):279--302, 2005.
	
	\bibitem{FMMM}
	Blas Fern{\'a}ndez, Roghayeh Maleki, {\v{S}}tefko Miklavi{\v{c}}, and Giusy
	Monzillo.
	\newblock Distance-regular graphs with classical parameters that support a
	uniform structure: case $ q\leq 1$.
	\newblock {\em arXiv preprint arXiv:2305.08937}, 2023.
	
	\bibitem{GoTer}
	Junie~T. Go and Paul Terwilliger.
	\newblock Tight distance-regular graphs and the subconstituent algebra.
	\newblock {\em European J. Combin.}, 23(7):793--816, 2002.
	
	\bibitem{hou2018folded}
	Li-hang Hou, Bo~Hou, and Suo-gang Gao.
	\newblock The folded $(2 d+ 1)$-cube and its uniform posets.
	\newblock {\em Acta Math. Appl. Sin. Engl. Ser.}, 34(2):281--292, 2018.
	
	\bibitem{JKT}
	Aleksandar Juri\v{s}i\'{c}, Jack Koolen, and Paul Terwilliger.
	\newblock Tight distance-regular graphs.
	\newblock {\em J. Algebraic Combin.}, 12(2):163--197, 2000.
	
	\bibitem{MikTer}
	\v{S}tefko Miklavi\v{c} and Paul Terwilliger.
	\newblock Bipartite {$Q$}-polynomial distance-regular graphs and uniform
	posets.
	\newblock {\em J. Algebraic Combin.}, 38(2):225--242, 2013.
	
	\bibitem{AN}
	Arnold Neumaier.
	\newblock Strongly regular graphs with smallest eigenvalue $-m$.
	\newblock {\em Archiv der Mathematik}, 33:392--400, 1979.
	
	\bibitem{No}
	Kazumasa Nomura.
	\newblock Homogeneous graphs and regular near polygons.
	\newblock {\em J. Combin. Theory Ser. B}, 60(1):63--71, 1994.
	
	\bibitem{Ter1}
	Paul Terwilliger.
	\newblock A new feasibility condition for distance-regular graphs.
	\newblock {\em Discrete Math.}, 61(2-3):311--315, 1986.
	
	\bibitem{OldTer}
	Paul Terwilliger.
	\newblock The incidence algebra of a uniform poset.
	\newblock In {\em Coding theory and design theory, {P}art {I}}, volume~20 of
	{\em IMA Vol. Math. Appl.}, pages 193--212. Springer, New York, 1990.
	
	\bibitem{Terpart1}
	Paul Terwilliger.
	\newblock The subconstituent algebra of an association scheme ({P}art {I}).
	\newblock {\em J. Algebraic Combin.}, 1(4):363--388, 1992.
	
	\bibitem{TerpartII}
	Paul Terwilliger.
	\newblock The subconstituent algebra of an association scheme (part II).
	\newblock {\em J. Algebraic Combin.}, 2:73--103, 1993.
	
	\bibitem{Ter}
	Paul Terwilliger.
	\newblock The subconstituent algebra of a distance-regular graph; thin modules
	with endpoint one.
	\newblock {\em Linear Algebra Appl.}, 356:157--187, 2002.
	
	\bibitem{inequ}
	Paul Terwilliger.
	\newblock An inequality involving the local eigenvalues of a distance-regular
	graph.
	\newblock {\em J. Algebraic Combin.}, 19:143--172, 2004.
	
	\bibitem{w:dual}
	Chalermpong Worawannotai.
	\newblock Dual polar graphs, the quantum algebra {$U_q\left(
		\mathfrak{{sl}}_2\right)$}, and {L}eonard systems of dual {$q$}-krawtchouk
	type.
	\newblock {\em Linear Algebra Appl.}, 438(1):443--497, 2013.
	
\end{thebibliography}
 \end{document}